\documentclass[a4paper]{article}
\usepackage[utf8]{inputenc}
\usepackage{amsmath}
\usepackage{amsfonts, amssymb}

\newtheorem{thm}{Theorem}
\newtheorem{lem}[thm]{Lemma}

\newtheorem{rem}[thm]{Remark}

\def\R{{\mathbb R}}
\def\E{{\mathbb E\,}}

\def\P{{\mathbb P}}

\newenvironment{proof}[1][] {\noindent {\bf Proof#1:} }{\hspace*{\fill}$\square$\medskip\par}

\def\AA{{\mathcal A}}
\def\eps{{\varepsilon}}
\def\d{{\,\mathrm d}}
\def\disdet{\Delta}
\def\condet{D}
\def\dissto{\Sigma}

\def\dis{\Xi}
\def\MA{{\mathcal{A}}}
\def\diag{\operatorname{diag}}

\def\V{\mathrm{Var}}

\def\wt{\widehat{t}}
\def\wv{\widehat{v}}

\providecommand{\keywords}[1]
{
  \small	
  \textbf{\textit{Keywords---}} #1
}

\title{Scaling limit of stretched Brownian chains}
\author{Frank Aurzada, Volker Betz, and Mikhail Lifshits}
\date{\today}

\begin{document}
\maketitle

\begin{abstract} 

We show that a properly scaled stretched long Brownian chain converges to a two-parametric stochastic process, given by the sum of an explicit deterministic continuous function and the solution of the stochastic heat equation with zero boundary conditions. 
\end{abstract}


\section{Introduction}

Interacting particle models are a popular tool to model catastrophic events in materials, 
for example the rupture of long polymer chains. For this problem, a relatively rich history of results and approaches exists.
The basic model in all cases is a chain of $d$ particles, modelled by their time-dependent  locations $X_t^1, \ldots, X_t^d$, and sometimes also by their time-dependent momenta. These particles are subjected to thermal noise and a 
nearest neighbour potential $V$ that generates an attractive force between particles with adjacent indices. One is then 
interested in the first time when one of the particle distances exceeds a given threshold, at which time one declares the chain 
to be broken.\\[1mm]
The case where the particle momenta are taken into consideration and thus the deterministic dynamics is Newtonian is rather difficult to treat, in particular when $V$ is not quadratic and 
thus the resulting Newton equations are not linear. This is due to the complexity of the resulting deterministic dynamical system. 
There are some early, non-rigorous works \cite{DT94,OT94} on this problem, and some more recent, rigorous ones 
\cite{MalMuz,Mal,Muz}, where however only the mean break time is investigated. In \cite{abh}, a very short chain is treated in an 
intermediate regime where the mass of the particles vanishes in the limit of small noise.\\[1mm]
The overdamped case is  more tractable: \cite{ABL1,ABL2,ABL3} rigorously study the detailed asymptotic distribution of the 
rupture time and position of finite chains, where the chain of particles is pulled apart at the right end and in the limit 
when both the speed of pulling and the variance of the noise vanish. In other cases, only non-rigorous results are availabe. 
In \cite{RBM19}, it is found that when the noise is not asymptotically vanishing, a chain of particles is `softer' in the middle, 
in the sense that it is more likely to break there than near the boundary. This is in contrast to the results of \cite{ABL1,ABL2,ABL3},
where it is found that in the case of asymptotically small noise, the break point distribution is uniform except at the two endpoints,
where it is half as likely to break. An interesting aspect of the break distribution is studied in \cite{CHHP21}: the authors find that 
under the condition that the chain equilibrates before it breaks (which corresponds to very slow pulling in the context of  
\cite{ABL1,ABL2,ABL3}),  the distribution of break times has a power law tail at the side of untypically short times, and this 
power law is independent of the length of the chain. This shows that there are still many interesting phenomena to be understood in 
the context of particle models of chain rupture, in particular in the limit of infinitely many particles.\\[1mm]
The present paper is a first step towards a rigorous study of interacting particle chains in the presence of outer forces, 
{\it in the limit of infinitely many particles}. We work in the framework adapted in \cite{ABL1}, where the forces between the chain elements 
are linear and the resulting process is thus Gaussian. In this setting, we identify the correct scaling of both the pulling speed 
and the noise, so that a non-trivial and well-defined limit process emerges. We prove the existence of a unique limiting process
consisting of a deterministic part (due to the pulling) and a mean zero random part, which represents the noise. The random component
of our limiting process is the solution of the stochastic heat equation with zero boundary conditions.\\[1mm]
The study of the convergence for classes of interacting particle processes that include the one we are treating 
(but without the pulling component) goes back at least to \cite{Fu83}, and has been carried out in \cite{Gy98} in great 
generality and for a strong type of convergence. We could have thus derived the {\it non-deterministic part} of our results from 
applying those theories. Nevertheless, we give a detailed proof of all our results. A first reason behind this is that our proofs are not
very long and including them makes the paper self-contained. Also, since we are working with Gaussian measures, 
our proof is different from, and significantly more transparent than, the general one, which may be a benenfit in itself. A second reason
is connected with the observation that our results show that the scaling that leads to the stochastic heat equation is not the correct one 
for studying the rupture of the chain. This is because, on the one hand, in the limiting equation, 
the roughness of the solution implies that any infinitesimal distance between two neighbouring particles
will be exceeded immediately. Thus there is no sensible break condition where the break time is greater than zero with positive probability.
More fundamentally, on the other hand, the interacting chain model relies on the fact that the order of the elements in the Brownian chain is conserved during 
the evolution - otherwise, the modelling assumption of nearest neighbour interaction, where neighbours are defined by the initial 
ordering of particles, becomes untenable. Although no explicit scaling ever appears in the
article \cite{CHHP21}, this means the scaling that the authors must have in mind when stating the independence of their results
from the length of the chain must be a different one from the one considered here. The investigation of different scalings leading to 
non-trivial rupture conditions is an interesting problem, and work in progress, and the detailed and rather explicit 
proof given here is a good reference point for finding the correct scalings and proving statements about the respective limits. \\[1mm]

This paper is structured as follows. In Section~\ref{sec:mainresults}, we collect our main results: We start with a precise specification of the model and then state the limiting result for the chain with pulling, when the number of particles tends to infinity. Afterwards, the result for the homogeneous system which is related to \cite{Gy98} and its relation to the stochastic heat equation is discussed. All proofs are collected in Section~\ref{sec:proofs}.

\section{Main results} \label{sec:mainresults}
\subsection{Brownian chains}
For every positive integer $d$ a stretched Brownian chain of $d+1$ particles is a solution of the following
system:
\begin{align}
 \d X_t^i &=(X_t^{i+1}-2X_t^{i}+X_t^{i-1}) \d t +  \sigma \d B^{i,d}_t, \quad i=1,\ldots,d-1; \notag
\\  \label{eqn:system0}
  X_t^0 &= 0, \qquad  X_t^d=d+ \eps\, t, \qquad X_0^i=i,\quad i=1,\ldots,d-1,
\end{align}
where $(B_t^{i,d})_{t\geq 0}$ are independent Brownian motions and $\eps,\sigma\geq 0$ are fixed.

Here the position of the leftmost particle $X_0$ is fixed, the rightmost particle is deterministically pulled to the right
with speed $\eps$, the nearest neighbours interact, and the particles are subject to independent Brownian motions.

This system was studied in \cite{ab,abh,ABL1,ABL2,ABL3} with a motivation to investigate chain rupture.
In these works, the number of particles was fixed and the system's behavior was studied as a function of the parameters
$\sigma$ and $\eps$. A natural question is {\it what happens when the number of particles grows to infinity?}
We show in this work that {\it after an appropriate scaling of $\sigma,\eps$ and time $t$ the stretched Brownian chain converges to a solution of the stochastic heat equation}.

At the first step we scale the parameters $\sigma$ and $\eps$ by passing to the system
\begin{align}
 \d X_t^i &=(X_t^{i+1}-2X_t^{i}+X_t^{i-1}) \d t + \sqrt{d} \sigma \d B^{i,d}_t, \quad i=1,\ldots,d-1; \notag
\\  \label{eqn:system1}
  X_t^0 &= 0, \qquad  X_t^d=d+\frac{\eps}{d}\, t, \qquad X_0^i=i,\quad i=1,\ldots,d-1.
\end{align}

Note that the pulling strength is now $\frac{\eps}{d}$, as compared to $\eps$ in \eqref{eqn:system0}. Similarly, instead of $\sigma$ we use $\sqrt{d}\sigma$ for the standard deviation of the noise. 

At the second step, we scale time and space in \eqref{eqn:system1} by considering the process $\dis\left(t,\frac{i}{d}\right):= d^{-1} X_{d^2t}^i$ which satisfies
\begin{align}
   \d \dis\left(t,\frac{i}{d}\right) &= d^2 \left( \dis\left(t, \frac{i+1}{d}\right) -2\dis\left(t, \frac{i}{d}\right)+\dis\left(t, \frac{i-1}{d}\right)
   \right) \d t +\sqrt{d}\sigma \d B_t^{i,d},\quad i=1,\ldots,d-1;
\nonumber
\\
  \dis(t,0)&= 0,\qquad \dis(t,1)=1+\eps t, \qquad \dis\left(0,\frac{i}{d}\right)=\frac id,\quad i=1,\ldots,d-1, \label{eqn:dissystem-re}
\end{align}
where $(B_t^{i,d})_{t\geq 0}$ are again independent Brownian motions (different from the ones in system \eqref{eqn:system1}).

Here we  deal again with $d+1$ interacting particles. 
The leftmost particle is still fixed, while the rightmost particle is pulled to the right at speed $\eps$. The main difference with \eqref{eqn:system1} is that
at time zero the particles are now equally distributed over $[0,1]$. 

The final scaling step is the passage to the continuous mass variable from the discrete one by letting
\begin{equation} \label{eqn:dissystem-tv}
    \dis_d(t,v):=\dis\left(t,\frac{\lfloor d v\rfloor}{d}\right), \qquad \textrm{for } t\geq 0, \ v\in[0,1].
\end{equation}

\subsection{Convergence to a limiting process}

Our aim is to show that, as $d\to\infty$, the processes  $\dis_d$ converge to an explicitly described limiting process $X$ solving the stochastic heat equation with certain initial conditions (uniform initial distribution of the mass) and certain boundary conditions (pulling the right end of the mass to the right at speed $\eps$). Namely,
\begin{align}   \label{continuous_problem_X}
   \d X_t   &=   \Delta X_t \d t +\sigma \d B_t,
\\  \nonumber
  X(t,0)&= 0,\qquad  X(t,1)=1+\eps t, \qquad X(0,v)=v,
\end{align}
where $(B_t)$
is an infinite-dimensional Brownian motion and $\Delta X_t = \partial^2_{v,v} X(t,v)$. We shall outline why our limiting process satisfies (\ref{continuous_problem_X}) in Section~\ref{sec:sheconnection}.

Let us comment on the type of convergence. We use a  type of convergence that is widely used in so called strong invariance principles for random walks (Strassen \cite{Str}, Koml\'os--Major--Tusn\'ady \cite{KMT1,KMT2}, M.\,Cs\"org\H{o}--R\'ev\'esz \cite{CR} et al). Namely, we construct the processes $\dis_d$ and the limiting process $X$ on a common probability space so that $\dis_d(t,v) \to X(t,v)$ uniformly on compacts with probability one.

The following theorem is the main result of the article.

\begin{thm} \label{thm:mainthm}
Fix $\eps,\sigma>0$. Let  $(B^k_t)_{t\geq 0}$, $k=1,2,\ldots$, be a sequence of independent Brownian motions.  Then there exists an array of Brownian motions 
$(B_t^{1,d},\ldots,B_t^{d-1,d})$ linearly depending on $(B^k)_{k\ge 1}$ and independent for each fixed $d$  such that
the solutions $\dis_d$ to the system  \eqref{eqn:dissystem-re} (in the form
\eqref{eqn:dissystem-tv}) converge, as $d\to\infty$, to the limiting process
\begin{equation} \label{eqn:DplusS}
   X(t,v):=D(t,v)+S(t,v), \qquad t\geq 0, v\in[0,1],
\end{equation}
where
\begin{equation} \label{eqn:D}
  D(t,v) := v(1+\eps t) + \eps\left[  h(v) -  \sum_{k=1}^\infty c_k e^{-\pi^2 k^2 t} \sqrt{2} \,\sin(k\pi v) \right]
\end{equation}
and
\begin{equation} \label{eqn:S}
     S(t,v) :=  \sigma \sum_{k=1}^\infty \int_0^t e^{-(t-u)\pi^2 k^2} \d B^k_u \cdot \sqrt{2} \, \sin(k\pi v)
\end{equation}
with $h(v):=\frac{1}{6}v(v^2-1)$ and $c_k:=\sqrt{2} \,\int_0^1 h(v)\sin(k\pi v)\d v$. 
Namely, for every fixed $T>0$ it is true that
\[
   \sup_{t\in[0,T], v\in[0,1]} |\dis_d(t,v)-X(t,v)| \to 0 
   \quad\textrm{a.s.}, \quad  \textrm{as } d\to\infty.
\]
\end{thm}

\begin{rem} {\rm
 The limiting process $(X(t,v))_{t\geq 0, v\in[0,1]}$ is  Gaussian and its expectation
$(D(t,v))_{t\geq 0, v\in[0,1]}$ is the limit of  $\dis_d$ in the deterministic case $\sigma=0$.

The process  $(S(t,v))_{t\geq 0, v\in[0,1]}$ is the limit of the solution of the `homogeneous system' with zero initial and zero boundary conditions. This will be explained in the next subsection.
}\end{rem}

\subsection{Convergence of the homogeneous systems}
The main idea behind our consideration is a splitting of a finite stretched Brownian chain's movement into a sum of two components. The first one is deterministic and takes into account the initial and boundary conditions. The second one solves the homogeneous stochastic equation
with zero initial and boundary conditions. Moreover, it admits a fairly explicit representation that we describe in this section. 
The two components lead to the terms $D$ and $S$ in \eqref{eqn:DplusS}, respectively.

In order to formulate the result, we need the following definition:
\begin{equation} \label{eqn:defnofA}
   \MA:= \begin{pmatrix}
   -2 & 1 & 0 & 0& \ldots& 0 \\
    1 & -2 & 1 & 0 &\ldots& 0 \\
    0 & 1& -2 & 1 & &  \\
      &   & \ddots & \ddots & \ddots  \\
    0 &\ldots & & 1 & -2 & 1 \\
    0 &\ldots & & 0 & 1 & -2
\end{pmatrix} \in \R^{(d-1)\times(d-1)}.
\end{equation}
This matrix is the interaction matrix of the particle systems \eqref{eqn:system0} and \eqref{eqn:system1}. Lemma~\ref{lem:eigenvectors} below shows that one can diagonalize  $\MA$ with the help of the eigenvectors and eigenfunctions: More precisely,
\begin{equation} \label{eqn:defnofQ}
\MA=Q\Lambda Q^\top, \qquad\text{with}\qquad \Lambda:=\diag(\lambda_1,\ldots,\lambda_{d-1}),  Q_{j,k}:=\sqrt{2/d}\cdot f_k^j,
\end{equation}
where
$$
\lambda_k:= - 2(1-\cos(k\pi/d)),\qquad f_k^m:=\sin(k m\pi/d),\quad m=1,\ldots,d-1, k=1,\ldots,d.
$$
It follows from Lemma~\ref{lem:eigenvectors} that $Q$ is an orthonormal matrix.

Now we are ready to state the representation result for finite $d$ and the corresponding
convergence result for homogeneous systems.

The following theorem is a substantial step for proving Theorem \ref{thm:mainthm}. Furthermore, we believe 
that it may be of an independent interest as a building block for eventual studies of similar systems with the boundary conditions different from ours.

\begin{thm} \label{thm:stolem2}
Fix $\sigma>0$. Let  $(B^k_t)_{t\geq 0}$, $k=1,2,\ldots$, be a sequence of independent Brownian motions. Consider the independent Brownian motions 
$(B_t^{1,d},\ldots,B_t^{d-1,d})^\top := Q (B_t^1,\ldots,B_t^{d-1})^\top$. 
The system
\begin{align}
    \d \dissto\left(t, \frac{i}{d}\right) &= d^2 \left( \dissto\left(t, \frac{i+1}{d}\right)-2\dissto\left(t, \frac{i}{d}\right)+\dissto\left(t, \frac{i-1}{d}\right)\right) \d t + \sqrt{d}\sigma \d B^{i,d}_t,\qquad i=1,\ldots,d-1; \label{eqn:system1sdes}
\\
     \dissto(t,0)&=0,\qquad \dissto(t,1)=0,\qquad \dissto\left(0,\frac{i}{d}\right)=0,\quad i=1,\ldots,d, \label{eqn:stochsystem0}
\end{align}
allows the following explicit solution
$$
    \dissto\left(t,\frac{i}{d}\right)=\sigma  \sum_{k=1}^{d-1} \int_0^t e^{ d^2 \lambda_k (t-u)} \d B_u^k \cdot \sqrt{2} \sin(k \pi i/d).
$$
Define $\dissto_d(t,v):=\dissto\left(t,\lfloor vd\rfloor/d\right)$. 
Then, for any $T>0$, 
\[
    \sup_{t\in[0,T], v\in[0,1]} |\dissto_d(t,v)-S(t,v)| \to 0
\]    
almost surely, where $S(t,v)$ is given by \eqref{eqn:S}.
\end{thm}

The proof of Theorem~\ref{thm:stolem2} is given in Lemmas~\ref{lem:asconv} and~\ref{lem:stolem2} below.

\subsection{Connections to stochastic heat equation}
\label{sec:sheconnection}

Let us now discuss connections between our limiting processes and the stochastic heat equation (SHE).

We start from the process $S$ appearing in Theorem \ref{thm:stolem2} and claim that $(S(t,v))$ is a solution to a formal stochastic heat equation
with trivial initial and boundary conditions
\begin{align}   \label{continuous_problem_S}
   \d  S_t   &=   \Delta S_t \d t +\sigma \d B_t,
\\  \nonumber
  S(t,0)&= 0,\qquad  S(t,1)=0, \qquad S(0,v)=0,
\end{align}
where $(B_t)$ is an infinite-dimensional Brownian motion.

To express this fact precisely, recall that the system of functions
$\psi_k(v):=\sqrt{2}\sin(\pi k v)$, with $k=1,2,...$, is an orthonormal base in $L_2[0,1]$ solving Dirichlet problem,
\begin{eqnarray*}
    \psi_k''(v) &=&-\theta_k\,\psi_k(v),
\\
     \psi_k(0)&=&\psi_k(1)=0,
\end{eqnarray*}
where the eigenvalues are given by $\theta_k=\pi^2k^2$.

Let us formally write the infinite-dimensional Brownian motion  taking values in $L_2[0,1]$ in a coordinate form
\begin{equation} \label{eqn:B_coordinate}
    B_t = \sum_{k=1}^\infty B_t^k \psi_k,
\end{equation}
where $(B_t^k)_{t\ge 0}$ are independent real-valued Brownian motions. 
Let us also represent the solution $S$ of the system \eqref{continuous_problem_S} in the same coordinate form:
\begin{equation} \label{eqn:S_coordinate}
    S(t,v)=  \sum_{k=1}^\infty  S^k(t)\psi_k(v).
\end{equation}
Then the system \eqref{continuous_problem_S} becomes
\begin{eqnarray*}
   \d S^k(t) &=& -\theta_k S^k(t) \d t+ \sigma \d B_t^k, \qquad k=1,2,...
\\
   S^k(0) &=& 0.
\end{eqnarray*}
This is a classical SDE solved by a non-stationary version of Ornstein--Uhlenbeck process
\[
   S^k(t) = \sigma \int_0^t e^{-(t-u)\theta_k} \d B_t^k.
\]
Now \eqref{eqn:S_coordinate} justifies  that $S$ given in \eqref{eqn:S} solves the system \eqref{continuous_problem_S}.

One should however stress that the representation \eqref{eqn:B_coordinate}  is rather formal because the sum does not converge for fixed $t$ and $v$. 
\medskip

Next, by using the properties of the base $(\psi_k)$ it is easy to check that the function $D$ defined in
\eqref{eqn:D} satisfies the deterministic heat PDE with boundary conditions:
\begin{align*} 
   \d D_t  &=   \Delta D_t \d t,
\\  \nonumber
  D(t,0)&= 0,\qquad  D(t,1)=1+\eps t, \qquad D(0,v)=v.
\end{align*}
Therefore, by summing up the results for $D$ and $S$, we see that the limiting process $X=D+S$ from Theorem 
\ref{thm:mainthm} solves the SPDE \eqref{continuous_problem_X}.


\section{Proofs} \label{sec:proofs}

\subsection{Outline}
We are going to show that the solution of the `deterministic system' converges to $D$. By `deterministic system', we mean the system (\ref{eqn:dissystem-re}) with $\sigma=0$. This is the purpose of Subsection~\ref{sec:detsystemcon}.
Further, we will show that the  the solution of `stochastic system' converges to $S$. By `stochastic system', we mean (\ref{eqn:system1sdes}) and (\ref{eqn:stochsystem0}).
This is the purpose of Subsections~\ref{sec:convsto}--\ref{ss:asconvergence}.
The solution of the system (\ref{eqn:dissystem-re}) is the sum of the `deterministic system' solution and that
of the `stochastic system', so that we will have shown the convergence of this sum to $D+S=X$. The two auxiliary sections \ref{sec:aux1} and \ref{sec:aux2} set up preparatory facts for deterministic and stochastic parts, respectively.

\subsection{Preliminaries for the determnistic case}
\label{sec:aux1}
We start with a few preparations. The first goal is to obtain the explicit eigenvalues of the matrix in (\ref{eqn:defnofA}).

\begin{lem}
\label{lem:eigenvectors}
The eigenvalues of $\MA$ defined in  \eqref{eqn:defnofA} are given by $\lambda_k:= - 2(1-\cos(k\pi/d))$ and the coordinates
of the corresponding eigenvectors $f_k$, $k=1,\ldots,d-1$, are given by $f_k^m:=\sin(k m\pi/d), m=1,\ldots,d-1$. Further, the vectors $\sqrt{2/d}\cdot f_k$, $k=1,\ldots,d-1$, are orthonormal.
\end{lem}

\begin{proof}
Set $g_k^m:=e^{i km\pi /d}$, $m=0,\ldots,d$. Then for $m=1,\ldots,d-1$, we have
\[
    g_k^{m+1}-2g_k^m + g_k^{m-1} = (e^{ik\pi/d} - 2 + e^{-ik\pi /d}) g_k^m =  - 2 (1-\cos(k\pi/d)) g_k^m = \lambda_k g_k^m.
\]
By taking the imaginary part of this identity, we obtain
\[
    \sin( k (m+1)\pi/d ) - 2\sin( km\pi/d )  +\sin( k (m-1)\pi/d )  = \lambda_k \sin( k m\pi/d ) .
\]
Taking into account that $\sin(0 k\pi/d)=0=\sin(d k\pi/d)$, we obtain the assertion $\MA f_k=\lambda_k f_k$.

 Since $f_k$ are the eigenvectors for distinct eigenvalues of a symmetric matrix, they are orthogonal.
 In order to see that $\sqrt{2/d}$ is the correct normalization,
just note that
\begin{eqnarray*}
  \sum_{m=1}^{d-1} (f_k^m)^2 &=&  \sum_{m=0}^{d-1} \sin(mk\pi/d)^2 = \frac{1}{2}\sum_{m=0}^{d-1} (1-\cos(2mk\pi/d))
\\
   &=&\frac{d}{2} - \frac{1}{2}\sum_{m=0}^{d-1} \cos(2mk\pi/d) = \frac{d}{2}\,.
\end{eqnarray*}

\end{proof}

\subsection{Convergence of the deterministic system solutions} \label{sec:detsystemcon}

We start with a representation of the solution of the deterministic system (i.e.\ system (\ref{eqn:dissystem-re}) with $\sigma=0$).

\begin{lem}
\label{lem:detsys1}
Let $(\disdet(t,\frac{i}{d}))_{i=0,\ldots,d;t\geq 0}$ be the solution of the system
\begin{align} \label{eqn:discreteproblem-re}
   \d \disdet(t,\frac{i}{d}) &= d^2 \left(
         \disdet\left(t, \frac{i+1}{d}\right) -2\disdet\left(t, \frac{i}{d}\right)+\disdet\left(t, \frac{i-1}{d}\right)
                                   \right) \d t,\quad i=1,\ldots,d-1;
\\  \nonumber
  \disdet(t,0)&= 0,\qquad \disdet(t,1)=1+\eps t,
  \qquad \disdet\left(0,\frac{i}{d}\right)=\frac id,\quad i=1,\ldots,d-1.
  \end{align}
Then the following explicit representation holds:
\begin{equation} \label{eqn:discrete_determ_solution}
  \disdet(t,\frac{i}{d}) = \frac{i}{d}(\eps t +1) + \eps \left( \frac{h^i}{d^2} -  \left[ e^{d^2 t \AA } \frac{1}{d^2} h \right]^i \right),
\end{equation}
where $h=(h^1,\ldots,h^d)^\top$, $h^i = \frac{i}{6d} (i^2-d^2)$.
\end{lem}

\begin{proof}
Let $(X^i_t)_{i=0,\ldots,d;t\geq 0}$ be the solution of the  system (\ref{eqn:system1}) with $\sigma=0$.
A representation of the solution of the system (\ref{eqn:system1}) is given in \cite[Lemma 4 and Lemma 5]{ABL1}.
One can easily see that $\disdet(t,\frac{i}{d}):=d^{-1}X^i_{d^2t}$ is given by \eqref{eqn:discrete_determ_solution}.
On the other hand, $\disdet(t,\frac{i}{d})$ is the solution of (\ref{eqn:discreteproblem-re}).
\end{proof}

\begin{lem} \label{lem:alternativedet}
Under the assumptions of Lemma~\ref{lem:detsys1}, we have the following alternative explicit representation:
\[
  \disdet(t,\frac{i}{d}) = \frac{i}{d}(\eps t +1) + \eps \left( \frac{h^i}{d^2} - \frac{2}{d^3}
  \sum_{k=1}^{d-1} \left[ \left( \sum_{m=1}^{d-1} h^m f_k^m\right) \cdot e^{d^2 t \lambda_k} \cdot f_k^i\right]  \right),
\]
where $h=(h^1,\ldots,h^d)^\top$, $h^i = \frac{i}{6d} (i^2-d^2)$, and $f_k^m=\sin(km\pi/d)$.
\end{lem}

\begin{proof}
We compute, using the diagonalization of $\MA$ introduced in (\ref{eqn:defnofQ}):

\begin{eqnarray*}
   \left[ e^{d^2 t \MA} h \right]^i
   = \left(e^{d^2 t \MA}  h \right)^\top e_i
   &=& \left(Q e^{d^2 t \Lambda} Q^\top   h \right)^\top e_i
\\
   &=&   h^\top Q e^{d^2 t \Lambda} Q^\top e_i
\\
   &=&  \sqrt{\frac{2}{d}}\, \sum_{k=1}^{d-1} h^\top Q e^{d^2 t \Lambda} f_k^i e_k
\\
  &=& \sqrt{\frac{2}{d}}\,\sum_{k=1}^{d-1} h^\top Q e_k e^{d^2 t \lambda_k} f_k^i
\\
  &=&\frac{2}{d} \sum_{k=1}^{d-1} h^\top f_k e^{d^2 t \lambda_k} f_k^i
\\
  &=&\frac{2}{d} \sum_{k=1}^{d-1} \left( \sum_{m=1}^{d-1} h^m f_k^m\right) e^{d^2 t \lambda_k} f_k^i,
\end{eqnarray*}
as claimed.
Here, we used $Q^\top e_i = \sqrt{2/d}\sum_{k=1}^{d-1} f_k^i e_k$ in the fourth step,
$Q e_k=\sqrt{2/d}\cdot f_k$ in the sixth step,
and $h^\top f_k=\sum_{m=1}^{d-1} h^m f_k^m$ in the seventh step.
\end{proof}

\begin{lem} \label{lem:seven}
The solution of the system (\ref{eqn:discreteproblem-re}) converges, as $d\to\infty$, to
\[
      \condet(t,v) = v(\eps t +1) + \eps \left( \frac{v}{6} (v^2-1)
      - \sum_{k=1}^{\infty}  c_k  \cdot e^{- k^2 \pi^2 t} \cdot \sqrt{2} \sin(k\pi v)  \right),
\]
where
\[
  c_k:=\int_0^1 \frac{x}{6} (x^2-1) \sqrt{2} \sin(x k \pi) \d x.
\]
\end{lem}

\begin{proof} We use the representation in Lemma~\ref{lem:alternativedet}.
Note that
\[
    \frac{1}{d^3} \sum_{m=1}^{d-1} h^m f_k^m = \frac{1}{d}\sum_{m=1}^{d-1} \frac{m}{6d}
    \left( \left( \frac{m}{d}\right)^2 -1 \right) \sin\left(\frac{m}{d} k \pi \right)
    \to \int_0^1 \frac{x}{6} ( x^2-1) \sin(xk\pi) \d x.
\]
Further,
\[
    d^2 \lambda_k = -2d^2 ( 1 - \cos( k\pi/d)) \to -k^2 \pi^2
\]
and
\[
    f_k^i =\sin(k i \pi / d) \to  \sin( k \pi v).
\]
\end{proof}
\subsection{Preliminaries for the stochastic case} \label{sec:aux2}
Here, we collect some more tools that are needed to prove the convergence of the stochastic system
solution.


\begin{lem} \label{lem:varest}
Let
$$
    w_k := \int_0^t e^{-k^2 \pi^2 (t-u)} \d B_u^k \qquad \text{and}\qquad w_{k,d}:=\int_0^t e^{d^2 \lambda_k (t-u)} \d B_u^k.
$$
Then there is an absolute constant $c>0$ such that
$$
     \V [w_k]\leq c k^{-2} \qquad \text{and}\qquad \V [w_{k,d}] \leq c k^{-2}.
$$
\end{lem}

\begin{proof} The proof of the first assertion is simple:
$$
     \V [w_k] = \int_0^t e^{-2k^2 \pi^2 (t-u) } \d u = \frac{1}{2 k^2 \pi^2} \left( 1 - e^{-2 k^2\pi^2 t}\right)
     \leq \frac{1}{2 k^2 \pi^2}.
$$
To see the second assertion, recall that there is a $c>0$ such that $\cos(u)\leq 1- c u^2$ for any $|u|\leq \pi$,
so that $-d^2\lambda_k=2 d^2 (1-\cos(\pi k/d))\geq 2c \pi^2 k^2$ for all $k$. This implies
$$
    \V [w_{k,d}] = \int_0^t e^{2 d^2 \lambda_k (t-u) } \d u
    = \frac{1}{-2 d^2 \lambda_k} \left( 1 - e^{2 d^2 \lambda_k t}\right) \leq \frac{1}{-2 d^2 \lambda_k}
    \leq \frac{1}{4c\pi^2k^2},
$$
for all $k$.
\end{proof}

We will also need the following large deviations result for Gaussian processes:

\begin{lem} \label{lem:ld-lemma}
 Let $\{Y_t, 0\le t \le T\}$ be a centered Gaussian process such that $Y_0=0$ and
for some $c>0,H\in(0,1]$
\[
  \E[(Y_{t_1}-Y_{t_2})^2] \le c |t_1-t_2|^{2H}, \qquad 0\le t_1,t_2\le T.
\]
Then there exist $c_1,c_2$ depending only on $c,H$ such that for all $r\ge 1$ it is true that
\[
   \P\left( \sup_{0\le t\le T} Y_t \ge c_1 r T^H \right) \le \exp\{-c_2 r^2\}.
\]
\end{lem}

\begin{proof}
 We may write the lemma's assumption as
\[
   \E[(Y_{t_1}-Y_{t_2})^2] \le c \, \E[(W^H_{t_1}-W^H_{t_2})^2], \qquad 0\le t_1,t_2\le T,
\]
where $W^H$ is an $H$-fractional Brownian motion. By the Sudakov--Fernique comparison principle \cite[p.190]{Lif95}, this inequality yields
\[
   \E \sup_{0\le t\le T} Y_t \le  \sqrt{c} \,  \E \sup_{0\le t\le T} W^H_t := \sqrt{c} \, q_H\, T^H.
\]
Let $m$ be a median of the r.v.\  $\sup_{0\le t\le T} Y_t$. Then
\[
   m \le  \E \sup_{0\le t\le T} Y_t \le  \sqrt{c} \, q_H\, T^H,
\]
where for the first inequality one can see \cite[p.143]{Lif95}.

On the other hand, under lemma's assumptions it is true that
\[
   \theta^2:= \sup_{0\le t \le T} \E Y_t^2 \le c \, T^{2H}.
\]

Now  for every $r\ge 1$ we obtain by the concentration principle \cite[p.141]{Lif95}
\begin{eqnarray*}
    \P\left( \sup_{0\le t\le T} Y_t \ge (2\sqrt{c}\,q_H)  r T^H \right)
    &\le&  \P\left( \sup_{0\le t\le T} Y_t  -m  \ge  (2\sqrt{c}\,q_H r - \sqrt{c}\,q_H) T^H  \right)
\\
    &=&  \P\left( \sup_{0\le t\le T} Y_t  -m  \ge   \sqrt{c}\,q_H r \, T^H  \right)
\\
    &\le& \bar\Phi \left( -  \sqrt{c} q_H r T^{H} / \theta  \right)
\\
    &\le& \bar\Phi \left( -  q_H r  \right)
    \le \exp( - q_H^2 r^2 /2 ),
\end{eqnarray*}
where $\bar\Phi(x)=\P(\mathcal{N}(0,1)>x)$, as required.
\end{proof}

\subsection{Convergence of the stochastic system}\label{sec:convsto}

In this subsection, we start with the description of the main result for the convergence of the stochastic system.

\begin{lem} \label{lem:asconv}
Let  $(B^k_u)_{u\geq 0}$, $k=1,2,\ldots$, be a sequence of independent Brownian motions. Consider
$$
  \dissto\left(t,\frac{i}{d}\right) :=  \sigma \sum_{k=1}^{d-1} \int_0^t e^{(t-u) d^2 \lambda_k} \d B^k_u  \ \cdot \  \sqrt{2}\sin(k \pi i/d)
$$
and
\[
  S(t,v) :=  \sigma \sum_{k=1}^\infty \int_0^t e^{-(t-u)\pi^2 k^2} \d B^k_u  \ \cdot \  \sqrt{2} \sin(k \pi v).
\]
Define $\dissto_d(t,v):=\dissto(t,\lfloor vd\rfloor/d)$. Then, for any $T>0$, 
\[
     \sup_{t\in[0,T], v\in[0,1]} |\dissto_d(t,v)-S(t,v)| \to 0
\]     
almost surely.
\end{lem}

We postpone the proof of this result to Subsection~\ref{ss:asconvergence}
and rather describe why Lemma~\ref{lem:asconv} yields Theorems~\ref{thm:mainthm} and
  \ref{thm:stolem2} : 

\begin{lem}
Let $(B_t^{1,d},\ldots,B_t^{d-1,d})^\top := Q (B_t^1,\ldots,B_t^{d-1})^\top$. Consider the system
\begin{align}
    \d \dissto\left(t, \frac{i}{d}\right) &= d^2 \left( \dissto\left(t, \frac{i+1}{d}\right)-2\dissto\left(t, \frac{i}{d}\right)+\dissto\left(t, \frac{i-1}{d}\right)\right) \d t + \sqrt{d}\sigma \d B^{i,d}_t,\qquad i=1,\ldots,d-1; \notag
\\
     \dissto(t,0)&=0,\qquad \dissto(t,1)=0,\qquad \dissto\left(0,\frac{i}{d}\right)=0,\quad i=1,\ldots,d. \label{eqn:stochsystem}
\end{align}
An explicit solution of the system \eqref{eqn:stochsystem} is given by
$$
     \dissto\left(t,\frac{i}{d}\right)=\sigma  \sum_{k=1}^{d-1} \int_0^t e^{ d^2 \lambda_k (t-u)} \d B_u^k \cdot \sqrt{2} \sin(ik \pi/d).
$$
\label{lem:stolem2}
    \end{lem}

Note that the representation of the stochastic system in Lemma~\ref{lem:stolem2} coincides with the term studied in Lemma~\ref{lem:asconv}. Therefore, Theorem~\ref{thm:stolem2} follows from Lemma~\ref{lem:asconv} and~\ref{lem:stolem2}. Finally, Theorem~\ref{thm:mainthm} follows from this and Lemma~\ref{lem:seven} for the deterministic part.

\medskip
\begin{proof}[ of Lemma~\ref{lem:stolem2}]
The first step is to write the solution  of the system (\ref{eqn:stochsystem}) as
$$
\dissto\left(t,\frac{i}{d}\right)=\sqrt{d}\sigma e^{d^2\MA t } \int_0^t e^{-d^2 \MA u} \d \bar B_u^d,
$$
where $\bar B_t^d = Q \sum_{k=1}^{d-1} e_k B^k_t$ and $(B^k_t)_{t\geq 0}$, $k=1,\ldots,d-1$, are the given independent Brownian motions.
This representation of the solution of the system (\ref{eqn:stochsystem}) is part of the statement of Lemma 6 in \cite{ABL1}.

We can then compute as follows:
\begin{align*}
\dissto\left(t,\frac{i}{d}\right)= \ & \sqrt{d}\sigma e^{d^2\MA t } \int_0^t e^{-d^2 \MA u}Q \d \bar B_u^d
\\
=\ & \sqrt{d}\sigma Q e^{d^2\Lambda t } Q^\top \int_0^t Q e^{-d^2 \Lambda u} Q^\top \sum_{k=1}^{d-1} Qe_k \d B_u^k
\\
=\ & \sqrt{d}\sigma\sum_{k=1}^{d-1} \int_0^t Q e_k \cdot e^{d^2\lambda_k(t-u)}  \d B_u^k
\\
=\ & \sqrt{d}\sigma\sum_{k=1}^{d-1} \int_0^t \sqrt{\frac{2}{d}} f_k \cdot e^{d^2\lambda_k(t-u)} \d B_u^k
\\
=\ & \sigma\sum_{k=1}^{d-1} \int_0^t  e^{d^2\lambda_k (t-u)} \d B_u^k\cdot \sqrt{2} f_k,
\end{align*}
as claimed.
  \end{proof}

%

\subsection{Quadratic mean evaluations}

The proof of Lemma~\ref{lem:asconv} will be based on several upper bounds for quadratic means that we collect here.

We use the following abbreviations
$$
     w_k:=w_k(t) := \int_0^t e^{-k^2 \pi^2 (t-u)} \d B_u^k,
     \qquad w_{k,d}:=w_{k,d}(t):=\int_0^t e^{d^2 \lambda_k (t-u)} \d B_u^k,
$$
$\phi_{k,d}:=\phi_{k,d}(\lfloor vd\rfloor /d):=\sin(k \pi \lfloor vd\rfloor /d)$,
and $\phi_k=\phi_k(v)=\sin(k \pi v)$.

In this notation,
$$
      S(t,v)=\sum_{k=1}^\infty w_k(t) \sqrt{2} \phi_k(v)\qquad\text{and}
      \qquad \dissto_d(t,v)=\sum_{k=1}^{d-1} w_{k,d}(t) \sqrt{2} \phi_{k,d}(\lfloor vd\rfloor/d).
$$

\bigskip

\paragraph*{Step 1: Quadratic mean convergence $\dissto_d(t,v)\to S(t,v)$ for any fixed $t\geq 0$ and $v\in[0,1]$.} ~ \\


Let us decompose
\begin{eqnarray*}
       \dissto_d(t,v)- S(t,v) &=& \sum_{k=1}^{d-1} w_{k,d} \sqrt{2} \phi_{k,d}
       - \sum_{k=1}^\infty w_k \sqrt{2} \phi_k
\\
       &=& \sum_{k=1}^{d-1} w_{k,d} \sqrt{2} (\phi_{k,d}-\phi_k)+ \sum_{k=1}^{d-1} (w_{k,d}-w_k) \sqrt{2} \phi_k
         - \sum_{k=d}^\infty w_k \sqrt{2} \phi_k
\\
&=:& S_1 + S_2 + S_3.
\end{eqnarray*}

In order to estimate the variance of $\dissto_d(t,v)- \dissto(t,v)$, we shall estimate
\begin{equation} \label{eqn:vare0}
    \V[\dissto_d(t,v)- \dissto(t,v)] = \V[S_1+S_2+S_3] \leq 3 (\V [S_1] +\V [S_2] +\V [S_3] )
\end{equation}
and control the variances of the $S_i$ individually.

First note that
$$
     \V[S_1] = \sum_{k=1}^{d-1} \V[w_{k,d}]\cdot 2 |\phi_{k,d}-\phi_k|^2.
$$
Using that $|\phi_{k,d}-\phi_k|=|\sin(k\pi \lfloor vd\rfloor /d ) - \sin(k\pi v)|
\leq |k\pi \lfloor vd \rfloor /d - k \pi v | \leq k\pi /d$ and the estimate from Lemma~\ref{lem:varest},
we obtain
\begin{equation} \label{eqn:step1vare1}
     \V[S_1] \leq 2 \sum_{k=1}^{d-1} c k^{-2} \cdot (k \pi /d)^2 = c' d^{-1}.
\end{equation}
Further,
$$
     \V[S_2] = \sum_{k=1}^{d-1} \V[w_{k,d}-w_k]\cdot 2 |\phi_k|^2 \leq  2 \sum_{k=1}^{d-1} \V[w_{k,d}-w_k].
$$
We observe that
$$
     \V[w_{k,d}-w_k] = \int_0^t |e^{d^2 \lambda_k u} - e^{-\pi^2 k^2 u} |^2 \d u.
$$
Recall that
\begin{eqnarray*}
      &&\int_0^\infty |e^{-r u} - e^{-s u}|^2 \d u = \int_0^\infty e^{-ru} |1 - e^{-(s-r) u}|^2 \d u
\\
     &\leq&  \int_0^\infty e^{-ru} (s-r)^2 u^2 \d u =(s-r)^2 c r^{-3}.
\end{eqnarray*}
Using this with $r=-d^2\lambda_k$, $s=\pi^2k^2$, and applying further that $|s-r|=|\pi^2 k^2 + d^2 \lambda_k|=d^2\lambda_k
+ \pi^2k^2=-2d^2(1-\cos(k\pi/d))+\pi^2k^2\leq c d^2 k^4 / d^4 = c k^4 / d^2$
(where we used $\cos x \leq 1 - x^2/2 + c x^4$ for a sufficiently large constant $c$ and all $x$)
and that $r=-d^2\lambda_k \geq c\pi^2 k^2$ (due to $\cos x \leq 1- cx^2$ for all $|x|\leq \pi$
and a sufficiently small constant $c>0$), we obtain that
\begin{equation} \label{eqn:step1vare2}
     \V[S_2]\leq   2 \sum_{k=1}^{d-1} \V[w_{k,d}-w_k] \leq c \sum_{k=1}^{d-1} (  k^4 / d^2 )^2 (k^2)^{-3}
     \leq c' d^3/d^4=c' d^{-1}.
\end{equation}
Finally, using Lemma~\ref{lem:varest}, we get
\begin{equation} \label{eqn:step1vare3}
       \V[S_3] = \sum_{k=d}^\infty \V[w_k] \cdot 2 \phi_k^2 \leq 2\sum_{k=d}^\infty \V[w_k] \leq c \sum_{k=d}^\infty k^{-2} = c' d^{-1}.
\end{equation}
Putting (\ref{eqn:vare0}), (\ref{eqn:step1vare1}), (\ref{eqn:step1vare2}), and (\ref{eqn:step1vare3}) together yields
\begin{equation} \label{eqn:step1estim}
      \V[ \dissto_d(t,v) - S(t,v)] \leq c d^{-1},\qquad \text{for all}~t\geq 0,v\in[0,1].
\end{equation}

\paragraph*{Step 2: Mean fluctuations in $v$}  ~\\

Fix $t>0$. We have, for any $v,v'\in[0,1]$,
\begin{eqnarray}   \nonumber
   \V[ S(t,v)-S(t,v') ]
   &=&  \V[ \sum_{k=1}^{\infty} w_k (\sin(\pi k v)-\sin(\pi k v')) ]
\\    \nonumber
    &=& \sum_{k=1}^{\infty} \V[w_k] |\sin(\pi k v)-\sin(\pi k v')|^2
\\   \nonumber
    &\le&  c \sum_{k=1}^{\infty} k^{-2} \min\{1, k^2 |v-v'|^2\}
\\  \nonumber
     &=&  c \sum_{k\le 1/ |v-v'|}^{\infty} k^{-2} k^2 |v-v'|^2  + c  \sum_{k> 1/ |v-v'| }^{\infty} k^{-2}
\\  \label{eqn:v-est}
   &\le& c |v-v'|.
\end{eqnarray}
We stress that the constant $c$ does not depend on $t$.

\paragraph*{Step 3: Mean fluctuations in $t$}  ~\\

Consider $t'>t>0$ and fix $v\in[0,1]$. Then we have
\begin{eqnarray}
     &&  \V[ S(t,v) - S(t',v)] \notag
\\
     &=& \sum_{k=1}^\infty \V[w_k(t)-w_k(t')]\cdot 2\cdot \phi_k^2  \notag
\\
     &\leq & 2 \sum_{k=1}^\infty \V[w_k(t)-w_k(t')]  \cdot 1  \notag
\\
     &=& 2 \sum_{k=1}^\infty \left[\int_0^t (1 - e^{-(t'-t)} )^2 e^{-(t-u)2\pi^2k^2} \d u
     + \int_t^{t'} e^{-(t'-u)2\pi^2 k^2 } \d u\right]. \label{eqn:silly}
\end{eqnarray}

Let us treat the two integrals separately.
\begin{eqnarray*}
     \int_0^t (1 - e^{-(t'-t)} )^2 e^{-(t-u)2\pi^2k^2} \d u   &\leq & (t'-t)^2 \int_0^t e^{-(t-u)2\pi^2k^2} \d u
\\
     &= &(t'-t)^2\int_0^t e^{-u2\pi^2k^2} \d u  \leq \frac{(t'-t)^2}{2\pi^2k^2}.
\end{eqnarray*}
Summing this in $k$ gives $c(t'-t)^2$. The second term in (\ref{eqn:silly}) equals
\begin{eqnarray*}
     2\sum_{k=1}^\infty \int_0^{t'-t}  e^{-2u \pi^2 k^2} \d u
     &= & 2 \sum_{k\leq (t'-t)^{-1/2}} \int_0^{t'-t}  e^{-2\pi^2 k^2 u}  \d u
\\
     && +  2\sum_{k> (t'-t)^{-1/2}} \int_0^{t'-t} e^{-2\pi^2 k^2 u}  \d u.
\end{eqnarray*}
In the first sum, we estimate the integrand by $1$ so that the whole sum is bounded by
$$
    2 (t'-t)^{-1/2} (t'-t) = 2 (t'-t)^{1/2}.
$$

The second sum is bounded by
$$
    2\sum_{k> (t'-t)^{-1/2}} \int_0^\infty e^{-2\pi^2 k^2 u}  \d u
    = 2\sum_{k> (t'-t)^{-1/2}} \frac{1}{2\pi^2 k^2} \leq c (t'-t)^{1/2}.
$$
This shows that
\begin{equation} \label{eqn:vartest}
      \V[ S(t,v) - S(t',v)] \leq c|t'-t|^{1/2}\qquad \text{for all}~t,t'\geq 0, v\in[0,1].
\end{equation}
We stress that the constant $c$ does not depend on $v$.


Using exactly the same arguments together with the fact that
$d^2\lambda_k \leq -c \pi^2 k^2$ for all $k$ and an appropriate constant $c>0$ --
we can obtain the following analog to estimate (\ref{eqn:vartest}):
\begin{equation} \label{eqn:varestdiffld}
      \V[ \dissto_d(t,v) - \dissto_d(t',v)] \leq c |t'-t|^{1/2}\qquad \text{for all}~t,t'\geq 0, v\in[0,1].
\end{equation}
Again, we stress that the constant $c>0$ does not depend on $v$.

\subsection{Joint continuity of $(S(t,v))$.}
From the estimates (\ref{eqn:v-est}) and (\ref{eqn:vartest}), it is straightforward to obtain the power estimate
\[
    \V[ S(t,v) - S(t',v') ]  \leq 2 c (  |t'-t|^{1/2}+ |v'-v|),
\]
for any $(t,v),(t',v') \in[0,T]\times[0,1]$.
Now the classical continuity criterion for Gaussian fields, see e.g. \cite[p.220]{Lif95},
implies that $(S(t,v))$ is almost surely continuous jointly in $t,v$ for $t,v\in[0,T]\times[0,1]$.
As the latter set is compact, $(S(t,v))$ is even uniformly continuous.


\subsection{Almost sure convergence (Proof of Lemma~\ref{lem:asconv})}
\label{ss:asconvergence}


Fix $T>0$. We shall consider the following time and space discretization:
\begin{eqnarray} \label{eqn:timedisc}
    t_j &:=& \frac{j T}{d}, \qquad j=0,\ldots,d.
\\
     v_i &:=& \frac{i}{d}, \qquad i=0,\ldots,d.
\end{eqnarray}
For $t\in [t_j,t_{j+1})$ we denote $\wt=\wt(t)=t_j$.
For $v\in[v_{i-1},v_i)$ let $\wv=\wv(v):=v_i$.

First notice that
\begin{eqnarray*}
   \dissto_d(t,v)-S(t,v)  &=& \dissto_d(t, \wv)-S(t,v)
\\
    &=& \dissto_d(t,\wv)-\dissto_d(\wt,\wv)
\\
    && ~+~\dissto_d(\wt,\wv)-S(\wt,\wv)
\\
    && ~ +~ S(\wt,\wv)-S(t,v).
\end{eqnarray*}
We will provide probabilistic bounds for the first two terms, while for the third term
we simply have
\begin{equation} \label{convergence3}
     \lim_{d\to\infty}  \sup_{(t,v)\in[0,T]\times[0,1]} |S(\wt,\wv)-S(t,v)| =0 \ \qquad \textrm{a.s.}
\end{equation}
by continuity of the process $S$, cf. the previous subsection.
\medskip

{\it A bound for the first term.}\  Note that
\[
   \sup_{(t,v)\in[0,T]\times[0,1]}  |\dissto_d(t,\wv)-\dissto_d(\wt,\wv)|
   =
   \max_{0\le j\le d-1}   \max_{0\le i\le d} \sup_{0\le \tau\le T/d}
     | \dissto_d(t_j+\tau, v_i)-\dissto_d(t_j, v_i)|.
\]
Fix $\eps>0$ and use Lemma~\ref{lem:ld-lemma} for
$Y_\tau:=\dissto_d(t_j+\tau,v_i)-\dissto_d(t_j,v_i)$ with $\tau\in[0,T/d]$.
The assumption of the lemma is verified by (\ref{eqn:varestdiffld}) with $H=1/4$.

Define $r$ via $r := \eps (d/T)^{1/4}/c_1$ where $c_1$ is the constant from Lemma~\ref{lem:ld-lemma}.
Then $c_1r(T/d)^{1/4}=\eps$  and $r\geq 1$ for $d$ large enough. This gives
\begin{eqnarray*}
 &&  \P\left( \sup_{0 \le \tau \le T/d}  | \dissto_d(t_j+\tau, v_i)-\dissto_d(t_j, v_i)| \ge \eps \right)
   =  \P\left( \sup_{0\le\tau\le T/d}  |Y_\tau| \ge \eps \right)
\\
   &\le& 2\,  \P\left( \sup_{0 \le \tau \le T/d}  Y_\tau \ge \eps \right)
   \le   2 \exp\left( -c_2 \eps^2 (d/T)^{1/2}/c_1^2  \right).
\end{eqnarray*}
It follows that
\[
   \P\left(   \sup_{(t,v)\in[0,T]\times[0,1]}  |\dissto_d(t,\wv)-\dissto_d(\wt,\wv)| \ge \eps \right)
   \le   (d+1)^2 \exp\left( -c_2 \eps^2 (d/T)^{1/2}/c_1^2  \right)
\]
and by the Borel--Cantelli lemma we obtain
\begin{equation} \label{convergence1}
     \lim_{d\to\infty}  \sup_{(t,v)\in[0,T]\times[0,1]}   |\dissto_d(t,\wv)-\dissto_d(\wt,\wv)|   =0 \ \qquad \textrm{a.s.}
\end{equation}
\medskip

{\it A bound for the second term.}\  Note that
\[
   \sup_{(t,v)\in[0,T]\times[0,1]} |\dissto_d(\wt,\wv)-S(\wt,\wv)|
   = \max_{0\le j\le d} \ \max_{0\le i\le d} |\dissto_d(t_j,v_i)-S(t_j,v_i)|.
\]
By using the variance bound \eqref{eqn:step1estim} for every pair $(i,j)$ we have
\[
   \P\left( |\dissto_d(t_j,v_i)-S(t_j,v_i)| \ge \eps \right)
   \le
   \exp\left( - \eps^2 d^2 / (2 c^2) \right),
\]
hence,
\[
   \P\left(  \sup_{(t,v)\in[0,T]\times[0,1]} |\dissto_d(\wt,\wv)-S(\wt,\wv)|\ge \eps\right)
   \le (d+1)^2 \exp\left( - \eps^2 d^2 / (2 c^2) \right)
\]
and by the Borel--Cantelli lemma we obtain
\begin{equation} \label{convergence2}
     \lim_{d\to\infty}  \sup_{(t,v)\in[0,T]\times[0,1]} |\dissto_d(\wt,\wv)-S(\wt,\wv)|
       =0 \ \qquad \textrm{a.s.}
\end{equation}
By combining \eqref{convergence3}, \eqref{convergence1}, and \eqref{convergence2}  we obtain
\[
    \lim_{d\to\infty}  \sup_{(t,v)\in[0,T]\times[0,1]}   |\dissto_d(t,v)-S(t,v)|
       =0 \ \qquad \textrm{a.s.},
\]
as claimed.


\begin{thebibliography}{10}


\bibitem{ab}
M.\ Allman  and V.\ Betz. Breaking the chain. {\it Stochastic Processes and Their Applications} {\bf 119} (2009), 2645--2659.

\bibitem{abh}
M.\ Allman, V.\ Betz, and M.\ Hairer. A chain of interacting
particles under strain. {\it Stochastic Processes and Their Applications} {\bf 121} (2011), 2014--2042.

\bibitem{ABL1}
F.\ Aurzada, V.\ Betz, and M.\ Lifshits. Breaking a chain of Brownian particles.
{\it Annals of Applied Probability} {\bf 31} (2021), 2585--2611.


\bibitem{ABL2}
F.\ Aurzada, V.\ Betz, and M.\ Lifshits.
Breaking a chain of interacting Brownian particles: A Gumbel limit theorem.
{\it Theory of Probability and Its Applications} {\bf 66} (2021), 84--208 (English).
{\it Teoriya Veroyatnostei i ee Primeneniya} {\bf 66} (2021), 231--260 (Russian).

\bibitem{ABL3}
F.\ Aurzada, V.\ Betz, and M.\ Lifshits.
Universal break law for chains of Brownian particles with nearest neighbour interaction.
{\it Journal of Physics A: Mathematical and Theoretical} {\bf 54} (2021), 305204.

\bibitem{CHHP21}  H.\ Charan, A.\ Hansen,  H.G.E.\ Hentschel, 
 and I.\ Procaccia. Aging and Failure of a Polymer Chain under Tension.  
{\em Phys. Rev. Lett.}  {\bf 126}  (2021), 085501


\bibitem{CR}
M.\ Cs\"org\H{o} and P. R\'ev\'esz. Strong Approximations in Probability and
Statistics. Springer, New York, 1981.

\bibitem{DT94} T.\ Doerr and P.\ Taylor.  Breaking in polymer chains. I. The
harmonic chain, {\it J. Chem. Phys.} {\bf 101} (1994), 10107.

\bibitem{Fu83} T.\ Funaki. Random motion of strings and related stochastic evolution equations. 
{\it Nagoya Mathematical Journal} {\bf  89} (1983), 129--193. 

\bibitem{Gy98} I.\ Gy\"ongy. Lattice Approximations for Stochastic Quasi-Linear Parabolic Partial 
Differential Equations Driven by Space-Time White Noise I. 
{\it Potential Analysis}  {\bf 9} (1998), 1--25. 

\bibitem{KMT1}
J.\ Koml\'os, P.\ Major, and G.\ Tusn\'ady. 
An approximation of partial sums of independent RV'-s and the sample DF. I. {\it Z. Wahrscheinlichkeitstheor. verw. Geb.} {\bf 32} (1975), 111--131.

\bibitem{KMT2}
J.\ Koml\'os, P.\ Major, and G.\ Tusn\'ady. 
An approximation of partial sums of independent RV'-s and the sample DF. II. 
{\it Z. Wahrscheinlichkeitstheor. verw. Geb.} {\bf 34} (1976), 34--58.

\bibitem{Lif95}
M.\ Lifshits. Gaussian Random Functions. Kluwer, Dordrecht, 1995.

\bibitem{MalMuz} V.A.\ Malyshev and S.A.\ Muzychka.
Dynamical phase transition in the simplest molecular chain model.
{\it Theoret. and Math. Phys.} {\bf 179}(1) (2014), 490--499.

\bibitem{Mal} V.A.\ Malyshev.
One-dimensional mechanical networks and crystals.
{\it Moscow Math. J.} {\bf 6}(2) (2006), 353--358.

\bibitem{Muz} S.A.\ Muzychk.
Mean exit time for a chain of $N = 2,3,4$ oscillators.
{\it Moscow University Math. Bull.} {\bf 68}(4) (2013), 206--210.

\bibitem{OT94} F.A.\ Oliveira and P.L.\ Taylor.  Breaking in polymer chains. II. The Lennard‐Jones chain. 
{\it J. Chem. Phys.} { \bf 101} (11) (1994), 10118--10125. 

\bibitem{RBM19} M.\ Razbin, P.\ Benetatos, and A.\ Moosavi-Movahedia.
A first-passage approach to the thermal breakage of a discrete one-dimensional chain.
{\it Soft Matter} {\bf 15} (2019),  2469.


\bibitem{Str}
V.\ Strassen. An invariance principle for the law of iterated logarithm.
{\it Z. Wahrscheinlichkeitstheor. verw. Geb.} {\bf 3} (1964), 211--226.
\end{thebibliography}
\end{document}